\UseRawInputEncoding 

\documentclass{article}
\usepackage{amsmath}
\usepackage{amssymb}
\usepackage{esint}
\usepackage{geometry}
\makeatletter
 \usepackage{amsfonts}
\usepackage{color}
\usepackage{graphicx}
\providecommand{\U}[1]{\protect \rule{.1in}{.1in}}

\newtheorem{theorem}{Theorem}[section]

\newtheorem{corollary}[theorem]{Corollary}

\newtheorem{definition}[theorem]{Definition}

\newtheorem{proposition}[theorem]{Proposition}
\newtheorem{remark}[theorem]{Remark}

\newenvironment{proof}[1][Proof]{\noindent \textbf{#1.} }{\  \rule{0.5em}{0.5em}}

\def\lt{\left}
\def\rt{\right}

\def\lu{\underline{\mu}}
\def\ou{\overline{\mu}}
\def\ls{\underline{\sigma}}
\def\os{\overline{\sigma}}

\def\vp{\varphi}

\def\br{\mathbb{R}}

\def\cp{\mathcal{P}}
\def\cf{\mathcal{F}}

\def\cf{\mathcal{F}}

\def\lu{\underline{\mu}}
\def\ou{\overline{\mu}}
\def\ls{\underline{\sigma}}
\def\os{\overline{\sigma}}

\def\be{\hat{\mathbb{E}}}

\def\bn{\mathbb{N}}

\def\vp{\varphi}

\makeatother

\begin{document}

\title{Maximum Likelihood Estimation for Maximal Distribution under Sublinear Expectation
\footnote{This work was supported by NSF of Shandong Provence (No.ZR2021MA018),  NSF of China (No.11601281),  National Key R\&D Program of China (No.2018YFA0703900) and the Young Scholars Program of Shandong University.}}

\author{Xinpeng Li\footnote{Corresponding author. Email: lixinpeng@sdu.edu.cn} \ \ Yue Liu \ \ Jiaquan Lu \\
Research Center for Mathematics and Interdisciplinary Sciences\\
 Shandong University, 266237, Qingdao, China
}

\date{ }

\maketitle

\abstract{Maximum likelihood estimation is a common method of estimating the parameters of the probability distribution from a given sample. This paper aims to introduce the maximum likelihood estimation in the framework of sublinear expectation. We find the maximum likelihood estimator for the parameters of the maximal distribution via the solution of the associated minimax problem, which coincides with the optimal unbiased estimation given by Jin and Peng \cite{JP21}. A general estimation method for samples with dependent structure is also provided. This result provides a theoretical foundation for the estimator of upper and lower variances, which is widely used in the G-VaR prediction model in finance.}



\section{Introduction}

The sublinear expectation theory established by Peng \cite{P19} is a powerful tool to deal with problems involving model uncertainties  in many fields, especially to solve dynamic problems with uncertainty in finance (see, for example, Epstein and Ji \cite{EJ}), in which the number of underlying probability measures may be infinite.

One typical distribution in sublinear expectation theory is the maximal distribution. It is usually used to characterize the worst case risk in finance, especially the uncertainty of returns of financial assets (see Li et al. \cite{lly} and Pei et al. \cite{Pei}). The primary advantage of maximally
distributed random variables for modelling purposes in applications is the simplicity of its calculation. For example, considering one-dimensional case, the distribution of maximally distributed random variable $X$ under the sublinear expectation $\be$ can be determined by two parameters $\lu$ and $\ou$, i.e.,
$$\be[\vp(X)]=\max_{\lu\leq\mu\leq\ou}\vp(\mu), \ \ \forall \vp\in C_b(\br).$$
It describes many real phenomena due to the law of large numbers with uncertainty, which is initialled by Peng \cite{P19} (see Theorem \ref{thm:LLN}).

A fundamental problem is how to choose suitable estimators of upper mean $\ou=\be[X]$ and lower mean $\lu=-\be[-X]$ for the maximally distributed random variable $X$? Recently, Jin and Peng \cite{JP21} finds that the largest unbiased estimator for $\ou$ and the smallest unbiased estimator for $\lu$ based on the independent maximally distributed samples $\{X_i\}_{i=1}^n$ can be calculated respectively by
\begin{equation}\label{eqeq1}
\hat{\ou}=\max\{X_1,\cdots,X_n\}, \ \ \hat{\lu}=\min\{X_1,\cdots,X_n\}.
\end{equation}
Based on these estimators, Peng et al. \cite{pyy} and Peng and Yang \cite{py} do extensive experiments on both the NASDAQ Composite Index and S\&P 500 Index and demonstrate the excellent performance of the $G$-VaR predictor, which is a non-trivial generalization of classical normal VaR model.

This paper provides a new perspective of these estimators based on the principle of maximum likelihood estimation (MLE) in the classical statistics theory (see, for example, Lehmann and Casella \cite{LC}). We propose a minimax problem in accordance with the essence of classical MLE. We maximize the ``probability'' of the samples with the smallest uncertainty, in which the additional minimum problem aims to reduce the uncertainty in the model. We find that our MLE for $\ou$ and $\lu$ coincides with Jin and Peng's optimal unbiased estimation (\ref{eqeq1}). In addition, our estimators are also valid for the dependent structure, and can be applied to approximate the samples unnecessarily maximally distributed. This new result provides the theoretical foundation for the estimator of upper and lower variances which is widely used in the $G$-VaR predictor model.

The remainder of this paper is organized as follows: in Section 2, we present some basic notions and results of sublinear expectation theory and the properties of maximal distribution. The detailed MLE of parameters for maximal distribution is provided in Section 3. In Section 4, we study the general estimator for the non-maximally distributed samples.

\section{Preliminaries of Sublinear Expectation Theory}
Let $\Omega$ be a Polish space and $\mathcal{H}$ be a linear space of real functions defined on $\Omega$ such that if $X_1,\cdots,X_n\in\mathcal{H}$ for each $n\in\mathbb{N}$, then $\varphi(X_1,\cdots,X_n)\in\mathcal{H}$, $\forall \varphi\in C_{Lip}(\mathbb{R}^n)$, where $C_{Lip}(\mathbb{R}^n)$ is the space of all Lipschitz functions on $\mathbb{R}^n$. 

\begin{definition} A sublinear expectation $\hat{\mathbb{E}}$ on $\mathcal{H}$ is a functional $\hat{\mathbb{E}}:\mathcal{H}\rightarrow\mathbb{R}$ satisfying the following conditions: $\forall X,Y\in\mathcal{H}$, we have
\begin{description}
  \item[(1)] Monotonicity: if $X\geq Y$, then $\hat{\mathbb{E}}[X]\geq\hat{\mathbb{E}}[Y]$;
  \item[(2)] Constant preserving: $\hat{\mathbb{E}}[c]=c$, $\forall c\in\mathbb{R}$;
  \item[(3)] Sub-additivity: $\hat{\mathbb{E}}[X+Y]\leq \hat{\mathbb{E}}[X]+\hat{\mathbb{E}}[Y]$;
  \item[(4)] Positive homogeneity: $\hat{\mathbb{E}}[\lambda X]=\lambda\hat{\mathbb{E}}[X]$, $\forall \lambda\geq0$.
\end{description}\label{def:sublinear expectation}
\end{definition}

The triple $(\Omega, \mathcal{H}, \hat{\mathbb{E}})$ is called the sublinear expectation space, which is analogous to the probability space $(\Omega,\cf,P)$.

One typical example of sublinear expectation is the upper expectation represented by
\begin{equation}\label{eqeq2}
\hat{\mathbb{E}}[X]=\sup_{P\in\mathcal{P}}E_P[X], \ \ \forall X\in\mathcal{H},
\end{equation}
where $\mathcal{P}$ is some set of probability measures on $(\Omega,\mathcal{B}(\Omega))$ and $E_P$ is the linear expectation introduced by $P$. The size of $\cp$ is used to characterize the uncertainty of model. In this case, the corresponding capacity introduced by $\cp$ can be defined as
$$V(A):=\sup_{P\in\cp}P(A), \ \ \forall A\in\mathcal{B}(\Omega).$$

The notions of identical distribution and independence are important in the classical probability theory and can also be non-trivially generalized to the framework of sublinear expectation theory in Peng \cite{P08,P19}.

\begin{definition}
Given an $n$-dimensional random vector $X=(X_1,\cdots,X_n)$ on a sublinear expectation space $(\Omega, \mathcal{H}, \hat{\mathbb{E}})$, where $X_i\in\mathcal{H},\ 1\leq i\leq n$, we define a functional on $C_{Lip}(\mathbb{R}^n)$ by
\[\hat{\mathbb{F}}_X[\varphi]:=\hat{\mathbb{E}}[\varphi(X)],~~\forall~ \varphi\in C_{Lip}(\mathbb{R}^n).\]
We call $\hat{\mathbb{F}}_X[\varphi]$ the sublinear distribution of $X$ under $\hat{\mathbb{E}}$.
\end{definition}
It is easy to see that $(\mathbb{R}^n, C_{Lip}(\mathbb{R}^n), \mathbb{F}_X)$ forms a sublinear expectation space.
{\remark
Given an integrable random variable $X$ on the classical probability space $(\Omega,\mathcal{F},P)$, we recall that the distribution function of $X$ is defined by
$$F_X(x)=P(X\leq x), \ \ \forall x\in\br.$$
For each $\vp\in C_{Lip}(\br)$, we can easily calculate
$$\mathbb{\hat{F}}_X[\vp]=E_P[\vp(X)]=\int_\br \vp(x)dF_X(x).$$
Conversely, if we know the value of $\mathbb{\hat{F}}_X[\vp]$ for every $\vp\in C_{Lip}(\br)$, then for each $x\in\br$, there exists a sequence of bounded and Lipschitz functions $$\vp_n(y)=\frac{1}{1+n(y-x)^+}$$ such that
$$\vp_n(y)\downarrow 1_{(-\infty,x]}(y),\ \ \forall y\in\br.$$
Then we obtain $\mathbb{\hat{F}}_X[\vp_n(X)]\downarrow F_X(x)$. Thus the distribution function $F_X$ is determined by $\mathbb{\hat{F}}_X$ in the linear case. But for the sublinear case, in particular, the sublinear expectation $\hat{\mathbb{E}}$ admits representation (\ref{eqeq2}), we emphasize that the following capacity
$$V(X\leq x):=\sup_{P\in\cp}P(X\leq x), \ \ \forall x\in\mathbb{R},$$
can not always determine the value of $\hat{\mathbb{E}}[\vp(X)]$. So we directly define $\be[\vp(X)]$ for each $\vp\in C_{Lip}(\br)$ as the distribution of $X$.
}

\begin{definition} Let $X$ and $Y$ be two $n$-dimensional random vectors defined on sublinear expectation spaces $(\Omega, \mathcal{H}, \hat{\mathbb{E}})$. They are called identically distributed if
\[\hat{\mathbb{E}}[\varphi(X)]=\hat{\mathbb{E}}[\varphi(Y)],~~\forall~ \varphi\in C_{Lip}(\mathbb{R}^n),\]
denoted by $X\mathop{=}\limits^d Y$.
\end{definition}

The following notion of independence provides a simple model of joint distribution $\hat{\mathbb{E}}[\vp(X,Y)]$ provided the marginal distributions.
\begin{definition}
\label{defind}Let $(\Omega, \mathcal{H}, \hat{\mathbb{E}})$ be a sublinear expectation space, an $n$-dimensional random vector $Y$ is said to be independent of another $m$-dimensional random vector $X$ under the sublinear expectation $\hat{\mathbb{E}}$, if $\forall~ \varphi\in C_{Lip}(\mathbb{R}^{m+n})$,
\begin{equation}\label{eqid}
\hat{\mathbb{E}}[\varphi(X,Y)]=\hat{\mathbb{E}}[\hat{\mathbb{E}}[\varphi(x,Y)]_{x=X}].
\end{equation}
Moreover, the sequence of random variables $\{X_i\}_{i=1}^\infty$ is said to be independent, if for each $i\geq 1$, $X_{i+1}$ is independent of $(X_1,\cdots,X_i).$
\end{definition}
{\remark\label{rem}
In order to explain the equation (\ref{eqid}), for simplicity, we only consider two random variables $X$ and $Y$, which are defined on the probability space $(\Omega,\cf,P)$ with the joint distribution function $F(x,y)=P(X\leq x, Y\leq y)$. If they are independent, then $F(x,y)=F_X(x)F_Y(y)$ for all $x,y\in\br$, where $F_X$ and $F_Y$ are distribution functions of $X$ and $Y$ respectively, we further have
\begin{align*}
E_P[\vp(X,Y)]&=\int_{\br^2}\vp(x,y)dF(x,y)\\
&=\int_{\br^2}\vp(x,y)dF_X(x)dF_Y(y)\\
&=\int_\br dF_X(x)\int_\br\vp(x,y)dF_Y(y)\\
&=\int_\br[E_P[\vp(x,Y)]dF_X(x)\\
&=E_P[E_P[\vp(x,Y)]|_{x=X}].
\end{align*}
Thus Definition \ref{defind} is the natural generalization of classical notion of independence. By Fubini's theorem, we obtain
$$E_P[\vp(X,Y)]=E_P[E_P[\vp(x,Y)]|_{x=X}]=E_P[E_P[\vp(X,y)]|_{y=Y}].$$
But it does not hold for sublinear expectation in general, the notion of independence under sublinear expectation is usually not symmetric, i.e., $Y$ being independent of $X$ can not automatically imply that $X$ is independent of $Y$. An interesting example can be found in Example 1.3.15 of Peng \cite{P19}. More properties of such independence under sublinear expectation and its relations with classical conditional expectations is referred to Guo et al. \cite{GLL}.
}
{\remark We note that $f(x):=\hat{\mathbb{E}}[\vp(x,Y)]$ may  be not continuous (resp. measurable) even if $\vp(x,y)$ is continuous (resp. measurable). Thus the joint distribution $\hat{\mathbb{E}}[\vp(X,Y)]:=\hat{\mathbb{E}}[f(X)]$ is not well-defined by the marginal distributions, since $\be$ is defined on the domain of continuous (resp. measurable) functions. So we consider the Lipschitz functions in the sublinear expectation theory.}

\begin{remark}
By (\ref{eqid}), it is obvious that for independent random variables $\{X_i\}_{i=1}^n$ and bounded Lipschitz functions $\vp_i\geq 0, 1\leq i\leq n$, we have
\begin{equation}\label{eqid2}
\be[\Pi_{i=1}^n\vp_i(X_i)]=\Pi_{i=1}^n\be[\vp_i(X_i)],
\end{equation}
which is equivalent to the classical independence when $\be$ is the linear expectation.

We also note that (\ref{eqid2}) is  weaker than  (\ref{eqid}). Moreover, (\ref{eqid2}) is an important property to obtain likelihood function in Section 3.
\end{remark}

Now we introduce the notion of maximal distribution, one of the fundamental sublinear distributions in the sublinear expectation theory.
{\definition Let $(\Omega, \mathcal{H}, \hat{\mathbb{E}})$ be a sublinear expectation space, an $n$-dimensional random vector $X$ is said to be maximally distributed if there exists a bounded, closed and convex subset $\Lambda\subset\mathbb{R}^n$ such that
\[\hat{\mathbb{E}}[\varphi(X)]=\mathop{\max}\limits_{x\in\Lambda}\varphi(x),~~\forall~ \varphi\in C_{Lip}(\mathbb{R}^{n}).\]}

For simplicity, we only consider one-dimensional case in this paper. More details about the maximal distribution, especially, the related maximally distributed random fields, can be found in Li and Peng \cite{LP22}.

The sublinear distribution of one-dimensional maximally distributed random variable $X$ is defined simply as
\begin{equation}\label{eqmd}
\hat{\mathbb{F}}_X[\varphi]=\hat{\mathbb{E}}[\varphi(X)]=\mathop{\max}\limits_{\underline{\mu}\leq x\leq \overline{\mu}}\varphi(x), \ \forall~ \varphi\in C_{Lip}(\mathbb{R}^n),
\end{equation}
where $\overline{\mu}:=\hat{\mathbb{E}}[X]$ and $\underline{\mu}:=-\hat{\mathbb{E}}[-X]$, denoted maximally distributed random variable $X$ by $X\mathop{=}\limits^d M_{[\underline{\mu},\overline{\mu}]}$.
The interval $[\underline{\mu}, \overline{\mu}]$ describes the uncertainty of the  sublinear distribution of $X$. Since such interval is bounded, (\ref{eqmd}) still holds for all continuous function $\vp$.

The following law of large numbers in Peng \cite{P19} plays an important role in the sublinear expectation theory.

\begin{theorem}[Law of large numbers] Let $\{X_i\}_{i=1}^\infty$ be an independent and identically distributed (i.i.d.) sequence of random variables defined on $(\Omega, \mathcal{H}, \hat{\mathbb{E}})$ and we further assume that $X_1$ is uniformly integrable under $\hat{\mathbb{E}}$, i.e.,
\begin{equation}\label{eqeq3}
\mathop{\lim}\limits_{\lambda\rightarrow\infty}\hat{\mathbb{E}}[(|X_1|-\lambda)^+]=0.
\end{equation}
Then for all $\varphi\in C_{Lip}(\mathbb{R})$, we have
\begin{equation}\label{wlln}
\mathop{\lim}\limits_{n\rightarrow\infty}\hat{\mathbb{E}}\left[\varphi\left(\frac{1}{n}\mathop{\sum}\limits_{i=1}^nX_i\right)\right]=\mathop{\max}\limits_{\mu\in[\underline{\mu},\overline{\mu}]}\varphi(\mu).
\end{equation}
\label{thm:LLN}
\end{theorem}
\begin{remark}
Recently, Fang et al. \cite{FPSS}, Song \cite{song} and Hu et al. \cite{hll} obtained the convergence rate of (\ref{wlln}) under higher moment conditions. If we further assume that $\be[X_1^2]<\infty$ in Theorem \ref{thm:LLN}, then we have
\begin{equation}\label{rate}
\lt|\be\lt[d^2_{[\lu,\ou]}\lt(\frac{\sum_{i=1}^nX_i}{n}\rt)\rt]\rt|\leq\frac{\be[X_1^2]}{n},
\end{equation}
where $d_{[\lu,\ou]}(x)=\inf_{y\in[\lu,\ou]}|x-y|$.

In addition, Chen \cite{chen} and Zhang \cite{zhang} established the corresponding strong law of large numbers.
\end{remark}

\begin{proposition} Let $X$ be a random variable defined on sublinear expectation spaces $(\Omega, \mathcal{H}, \hat{\mathbb{E}})$, we further assume that $X$ is uniformly integrable under $\hat{\mathbb{E}}$. Then $X$ is maximally distributed if and only if
\begin{equation}\label{max}
aX+b\bar{X}\mathop{=}\limits^d(a+b)X, \ \ \forall a,b\geq 0.
\end{equation}
where $\bar{X}$ is an independent copy of $X$, i.e., $\bar{X}$ is independent of $X$ and $\bar{X}\overset{d}{=}X$. \label{md}
\end{proposition}

\begin{proof}
Let $X\overset{d}{=}M_{[\underline{\mu},\overline{\mu}]}$ and $\Lambda=[\underline{\mu},\overline{\mu}]$, then we have,
$\forall \vp\in C_{Lip}(\br),$
\begin{align*}
\hat{\mathbb{E}}[\varphi(aX+b\bar{X})]&=
\hat{\mathbb{E}}[\hat{\mathbb{E}}[\varphi(ax+b\bar{X})]_{x=X}]\\
&=\mathop{\max}\limits_{x\in\Lambda} \mathop{\max}\limits_{\bar{x}\in\Lambda}\varphi(ax+b\bar{x})\\
&=\mathop {\max }\limits_{\mu\in\Lambda}\varphi [(a+b)\mu]\\
&=\hat{\mathbb{E}}[\varphi((a+b)X)], \
\end{align*}
thus (\ref{max}) holds.

Conversely, we construct an i.i.d. sequences $\{X_i\}_{i=1}^{\infty}$ with $X_1\mathop{= }\limits^d X$ and define
\[\eta_n:=\frac{1}{2^n}(X_1+X_2+\cdots+X_{2^n}), \ \ \forall n\in\bn\]
In particular, taking $a=b=1$ in (\ref{max}), we have
\begin{align*}
X_1+X_2&\mathop{=}\limits^d 2X_1,\\
&\cdots\mathrm{}\\
X_{2^n-1}+X_{2^n}&\mathop{=}\limits^d 2X_{2^n-1}.
\end{align*}
By induction, we obtain, for each $n\in\bn$,
\begin{align*}
\eta_n&\overset{d}{=}\frac{1}{2^n}(2X_1+2X_3+\cdots+X_{2^n-1})\\
&\mathop{=}\limits^d \frac{1}{2^{n-1}}(X_1+X_3+\cdots+X_{2^n-1})\\
&\mathop{=}\limits^d\cdots\mathop{=}\limits^dX_1\overset{d}{=}X.
\end{align*}
By Theorem \ref{thm:LLN}, we have
\[\mathop{\lim}\limits_{n\rightarrow\infty}\hat{\mathbb{E}}[\varphi(\eta_n)]=\mathop{\max}\limits_{\mu\in\Lambda}\varphi(\mu),\]
which implies that
$$\be[\vp(X)]=\max_{\mu\in\Lambda}\vp(\mu).$$
Hence $X$ is maximally distributed.
\end{proof}

{\corollary An uniformly integrable random variable $X$ on $(\Omega,\mathcal{H},\hat{\mathbb{E}})$ is maximally distributed if and only if
\[X+\bar{X}\mathop{=}\limits^{d}2X,\]
where $\bar{X}$ is the independent copy of $X$.}

{\remark
A counterexample in Guo and Li \cite{GL} shows that the law of large numbers fails without the uniformly integrable condition (\ref{eqeq3}). It is still open that whether Proposition \ref{md} still holds without such integrable condition.
}

\section{Maximum Likelihood Estimation for Independent Samples}

The idea of MLE is to find proper parameters to maximize the probability $P(X_1=x_1,\cdots,X_n=x_n)$ of realized samples $\{x_i\}_{i=1}^n$ of population $X$ which has prescribed probability measure $P$. Analogously, for the maximal distribution, we hope to maximize the following capacity and call it the likelihood function,
$$V(x_1,\cdots,x_n):=\hat{\mathbb{E}}[I_{\{{X_1=x_1,\cdots,X_n=x_n}\}}].$$
It is worth pointing out that $I_{\{{X_1=x_1,\cdots,X_n=x_n}\}}\notin\mathcal{H}$ since indicator function is not continuous. But it can be well-defined by the fact that such indicator function can be approximated by the Lipschitz functions (see Theorem \ref{indd}).

Moreover, if $\be$ can be represent as
$$\be[\cdot]=\sup_{P\in\cp}E_P[\cdot],$$
then it is natural to define
$$V(x_1,\cdots,x_n):=\sup_{P\in\cp}P(X_1=x_1,\cdots,X_n=x_n).$$

In particular, for the maximally distributed population $X$ with parameters $\lu$ and $\ou$, we denote the corresponding likelihood function as
$$V(x_1,\cdots,x_n;\lu,\ou).$$
Obviously, the value of likelihood function is increasing when the interval $[\underline{\mu},\overline{\mu}]$ is enlarging.

Let $\Delta=\overline{\mu}-\underline{\mu}$ be the degree of uncertainty of maximal distribution, and we also hope to deduce the uncertainty when we maximize the likelihood function, thus the MLE of parameters $\underline{\mu}$ and $\overline{\mu}$ is to solve the following minimax problem:
\begin{equation}\label{mle}
\min_{\Delta}\max_{\underline{\mu},\overline{\mu}} V(x_1,\cdots,x_n;\underline{\mu},\overline{\mu}).
\end{equation}

In order to solve such minimax problem, we firstly establish a representation theorem for maximal distribution by the Dirac measures, where the Dirac measure on a point $\mu\in\mathbb{R}$ is denoted by $\delta_\mu$ satisfying
\begin{equation*}
\delta_\mu(A)=\left\{
\begin{aligned}
&1,&\mu\in A,\\
&0,&\mu\notin A,
\end{aligned}
\right.\ \ \ \ \ \forall A \in\mathcal{B}(\mathbb{R}).
\end{equation*}

{\theorem Let   $(\Omega,\mathcal{H}, \hat{\mathbb{E}})$ be a sublinear expectation space and $X\overset{d}{=}M[\underline{\mu},\overline{\mu}]$, then for each $\varphi\in C_{Lip}(\mathbb{R})$,
\begin{align*}\label{eq:equ}
\hat{\mathbb{E}}[\varphi(X)]&=\mathop{\max}\limits_{\underline{\mu}\leq x\leq\overline{\mu}}\varphi(x)\\
&=\mathop{\max}\limits_{\underline{\mu}\leq \mu\leq\overline{\mu}}\int_{\underline{\mu}}^{\overline{\mu}}\varphi(y)\delta_\mu(dy).
\end{align*}
\label{repp}}

\begin{proof}
On one hand, it is clear that
\begin{align*}
 \mathop{\max}\limits_{\underline{\mu}\leq \mu\leq\overline{\mu}}\int_{\underline{\mu}}^{\overline{\mu}}{\varphi(y)}\delta_\mu(dy)&\leq\mathop{\max}\limits_{\underline{\mu}\leq \mu\leq\overline{\mu}}\mathop{\max}\limits_{\underline{\mu}\leq y\leq\overline{\mu}}\varphi(y)\int_{\underline{\mu}}^{\overline {\mu}} \delta_\mu(dy)\\
 &=\mathop{\max}\limits_{\underline{\mu}\leq y\leq \overline{\mu}} \varphi(y).
\end{align*}
On the other hand, there exists $\mu^*\in[\underline{\mu},\overline{\mu}]$ such that $\varphi(\mu^*)=\mathop{\max}\limits_{\underline{\mu}\leq x\leq \overline{\mu}}\varphi(x)$. Then we have
\begin{align*}{\max}_{\underline{\mu}\leq \mu\leq\overline{\mu}}\int_{\underline{\mu}}^{\overline{\mu}}{\varphi(y)}\delta_\mu(dy)&\geq\int_{\underline{\mu}}^{\overline{\mu}}{\varphi(y)}\delta_{\mu^*}(dy)\\
&=\varphi(\mu^*)=\mathop{\max}_{\underline{\mu}\leq x\leq\overline{\mu}}\varphi(x).
\end{align*}
\end{proof}

{\remark Let $\{\delta_{\mu_n}\}_{n=1}^\infty$ be a sequence of Dirac measures with $\mu_n\in[\underline{\mu},\overline{\mu}]$. Then there exists a point $\mu^*\in[\underline{\mu},\overline{\mu}]$ and subsequence $\{\delta_{{\mu_{n_i}}}\}_{i=1}^\infty$ such that $\delta_{\mu_{n_i}}$ weakly converges to $\delta_{\mu^*}$ provided $\mu_{n_i}\to\mu^*$. Thus the set of Dirac measures $\mathcal{P}=\{\delta_\mu, \mu\in[\underline{\mu},\overline{\mu}]\}$ is weakly compact. \label{remark:C_blip}}

Secondly, by the independence of samples, we can simply calculate the likelihood function.

\begin{theorem}\label{indd}
Let $\{X_i\}_{i=1}^n$ be a sequence of independent random variables on sublinear expectation space $(\Omega,\mathcal{H},\hat{\mathbb{E}})$. We further assume that $\hat{\mathbb{E}}$ can be represented by
$$\be[X]=\max_{P\in\mathcal{P}}E_P[X],  \ \ \forall X\in\mathcal{H},$$
where $\mathcal{P}$ is a weakly compact set of probability measures on  $(\Omega,\mathcal{B}(\Omega))$.

Then the terms $\hat{\mathbb{E}}[I_{\{X_1=x_1,\cdots,X_n=x_n\}}]$ and $\be[I_{\{X_i=x_i\}}]$, $1\leq i\leq n$ are well-defined.

Furthermore, $\forall x_i\in\mathbb{R}, \ 1\leq i\leq n.$
\begin{equation}\label{ind}
\hat{\mathbb{E}}[I_{\{X_1=x_1,\cdots,X_n=x_n\}}]=\Pi_{i=1}^n\hat{\mathbb{E}}[I_{\{X_i=x_i\}}].
\end{equation}
\end{theorem}
\begin{proof}
The indicator function $I_{x^*}(\cdot)$, $x^*\in\mathbb{R}$ is not continuous, thus $I_{\{X_i=x_i\}}\notin\mathcal{H}$. But it can be approximated by the Lipschitz functions, thus $\hat{\mathbb{E}}[I_{\{X_i=x_i\}}]$ can be well-defined, so does $\hat{\mathbb{E}}[I_{\{X_1=x_1,\cdots,X_n=x_n\}}]$.

Indeed, for fixed $x^*\in\mathbb{R}$, consider the function
$$\varphi_k^{x^*}(x)=\frac{1}{1+k|x-x^*|}.$$
It is easily seen that $\varphi_k^{x^*}(\cdot)$ is a non-negative bounded continuous function and $\varphi_k^{x^*}(\cdot)\downarrow I_{x^*}(\cdot)$.

For each $k\in\mathbb{N}$, we also have
\begin{align*}
|\varphi_k^{x^*}(x_1)-\varphi_k^{x^*}(x_2)|
&=\left|\frac{1+k|x_2-x^*|-1-k|x_1-x^*|}{(1+k|x_1-x^*|)(1+k|x_2-x^*|)}\right|\\
&\leq \frac{k|x_2-x_1|}{(1+k|x_1-x^*|)(1+k|x_2-x^*|)}\\
&\leq k|x_2-x_1|.
\end{align*}
Thus $\varphi_k^{x^*}$ is a Lipschitz function.

Then by Theorem 31 in Denis et al. \cite{DHP11}, we obtain
$$\be[\vp^{x^*}_k(X_i)]\downarrow\be[I_{\{X_i=x^*\}}].$$
Similarly, the term $\hat{\mathbb{E}}[I_{\{X_1=x_1,\cdots,X_n=x_n\}}]$ is well-defined for each $n\in\bn$, which can be approximated as
$$\be[\Pi_{i=1}^n\vp^{x_i}_k(X_i)]\downarrow\hat{\mathbb{E}}[I_{\{X_1=x_1,\cdots,X_n=x_n\}}].$$
By the independence of $\{X_i\}_{i=1}^n$, we obtain, $\forall x_i\in\mathbb{R}, 1\leq i\leq n$,
\begin{align*}
\hat{\mathbb{E}}[I_{\{X_1=x_1,\cdots,X_n=x_n\}}]&=\lim_{k\to\infty}\hat{\mathbb{E}}[\Pi_{i=1}^n\vp_k^{x_i}(X_i)]\\
&=\lim_{k\to\infty}\Pi_{i=1}^n\be[\vp_k^{x_i}(X_i)]\\
&=\Pi_{i=1}^n\lim_{k\to\infty}\be[\vp_k^{x_i}(X_i)]\\
&=\Pi_{i=1}^n\hat{\mathbb{E}}[I_{\{X_i=x_i\}}].
\end{align*}
\end{proof}

By Theorem \ref{indd} and  \ref{repp}, we can see that if $X\overset{d}{=}M[\lu,\ou]$, then
$$\be[I_{\{X=x^*\}}]=\max_{\lu\leq\mu\leq\ou}\delta_\mu(x^*).$$

Now we can solve the minimax problem (\ref{mle}).

\begin{theorem}\label{thmle}
Let $\{X_i\}_{i=1}^n$ be an i.i.d. sequence with $X_1\overset{d}{=}M_{[\underline{\mu},\overline{\mu}]}$. Then the MLE of parameters $\underline{\mu}$ and $\overline{\mu}$ is given by
\begin{equation}\label{est}
\hat{\underline{\mu}}=\min\{X_1,\cdots, X_n\}; \ \ \hat{\overline{\mu}}=\max\{X_1,\cdots,X_n\}.
\end{equation}
\end{theorem}
\begin{proof}
Let $\{x_i\}_{i=1}^n$ be the samples of population $\{X_i\}_{i=1}^n$.

By Theorem \ref{repp} and \ref{indd}, we have
\begin{align*}
V(x_1,\cdots,x_n;\underline{\mu},\overline{\mu})&=\hat{\mathbb{E}}[I_{\{X_1=x_1,\cdots,X_n=x_n\}}]\\
&=\Pi_{i=1}^n\hat{\mathbb{E}}[I_{\{X_i=x_i\}}]\\
&=\Pi_{i=1}^n\max_{\underline{\mu}\leq\mu\leq\overline{\mu}}\delta_\mu(x_i).
\end{align*}
It is clear that
\begin{equation}\label{eqv}
V(x_1,\cdots,x_n;\underline{\mu},\overline{\mu})=\left\{
\begin{aligned}
&1, \ \ \underline{\mu}\leq\underline{x}\leq\overline{x}\leq\overline{\mu},\\
&0, \ \ \ \text{otherwise},
\end{aligned}
\right.
\end{equation}
where $\underline{x}=\min\{x_1,\cdots,x_n\}$ and $\overline{x}=\max\{x_1,\cdots,x_n\}$.

Thus the solution of minimax problem (\ref{mle}) is given by
$$\hat{\underline{\mu}}=\min\{X_1,\cdots,X_n\}, \ \ \ \hat{\overline{\mu}}=\max\{X_1,\cdots,X_n\}.$$
\end{proof}

\begin{remark}
The MLE in (\ref{est}) is the same as the largest unbiased estimator for $\overline{\mu}$ and smallest unbiased estimator for $\underline{\mu}$ obtained in Jin and Peng \cite{JP21}. In which, a statistic $T_n=f_n(X_1,\cdots,X_n)$ is called an unbiased estimator of parameter $\mu$ if $\be[f_n(X_1,\cdots,X_n)]=\mu$, where $f_n$ is continuous  on $\br^n$.
\end{remark}

\section{General Maximum Likelihood Estimator}

In this section, we firstly prove that Theorem \ref{thmle} still holds without the assumption of independence, which indicates that our results can also be applied to the samples with dependent structure.

\begin{theorem}\label{thmle2}
Let $\{X_i\}_{i=1}^n$ be identically distributed sequence with $X_1\overset{d}{=}M_{[\underline{\mu},\overline{\mu}]}$. Then the MLE of parameters $\underline{\mu}$ and $\overline{\mu}$ is given by
\begin{equation}\label{est2}
\hat{\underline{\mu}}=\min\{X_1,\cdots, X_n\}; \ \ \hat{\overline{\mu}}=\max\{X_1,\cdots,X_n\}.
\end{equation}
\end{theorem}
\begin{proof}
Let $\{x_i\}_{i=1}^n$ be the sample of population $\{X_i\}_{i=1}^n$.
Obviously, $\forall\  0\leq\vp_i\in C_{Lip}(\br), \  1\leq i\leq n,$ we have
$$\be[\Pi_{i=1}^n\vp_i(X_i)]\leq\Pi_{i=1}^n\max_{\lu\leq \mu_i\leq\ou}\vp_i(\mu_i).$$
By the similar argument in the previous proofs, we obtain
$$\be[\Pi_{i=1}^nI_{\{X_i=x_i\}}]\leq\Pi_{i=1}^n\be[I_{\{X_i=x_i\}}].$$

Then we have
\begin{align*}
\bar{V}(x_1,\cdots,x_n;\underline{\mu},\overline{\mu}):&=\be[I_{\{X_1=x_1,\cdots,X_n=x_n\}}]\\
&=\be[\Pi_{i=1}^nI_{\{X_i=x_i\}}]\\
&\leq\Pi_{i=1}^n\be[I_{\{X_i=x_i\}}]\\
&={V}(x_1,\cdots,x_n;\underline{\mu},\overline{\mu}),
\end{align*}
where ${V}(x_1,\cdots,x_n;\underline{\mu},\overline{\mu})$ is defined as in (\ref{eqv}).

Since the sample $\{x_i\}_{i=1}^n$ is realized, we have
$$0<\bar{V}(x_1,\cdots,x_n;\underline{\mu},\overline{\mu})\leq 1,$$
thus
$${V}(x_1,\cdots,x_n;\underline{\mu},\overline{\mu})=1,$$
which implies that
$$\hat{\underline{\mu}}=\min\{x_1,\cdots,x_n\}, \ \ \ \hat{\overline{\mu}}=\max\{x_1,\cdots,x_n\}.$$
\end{proof}

In many practical situations, we often have i.i.d. condition for samples $\{X_i\}_{i=1}^n$ of population $X$  which is not maximally distributed, and we hope to estimate the value of $\be[\vp(X)]$ for some $\vp$ based on the i.i.d. samples $\{X_i\}_{i=1}^n$. Thanks to the law of large numbers, our results can also be applied to approximate the non-maximally distributed population.

One typical application is to estimate the upper and lower variance of population $Z$ which has model  uncertainty based on the historical time series $\{Z_{t-i}\}_{i\geq 1}$   (see Li et al. \cite{lly}, Peng et al. \cite{pyy} and Peng and Yang \cite{py}).

We assume that $Z$ has mean-zero, i.e.,
$$\be[Z]=\be[-Z]=0.$$
The upper and lower variance of $Z$, defined as $\os^2=\be[Z^2]$ and $\ls^2=-\be[-Z^2]$, can be estimated as follows:

Taking  $L\in\bn$ to be large enough, and we calculate the local sample variance with window $L$ by
$$\sigma^2_{j}=\frac{\sum_{i=1}^L(Z_{t-L+i-j}-\mu_j)^2}{L-1}, \ 1\leq j\leq K,$$
where $\mu_j=\frac{\sum_{i=1}^LZ_{t-L+i-j}}{L}$ is the sample mean under the same window $L$ and the number $K\in\bn$ is fixed with prior knowledge. The value of $K$ can be used to characterize the uncertainty of model. The larger $K$ means that we prefer more uncertainty in the model.

We note that the term $\mu_j$ is closed to 0 when $L$ is large enough, since $\be[Z]=\be[-Z]=0$. Thus we have
$$\sigma_j^2\approx\frac{\sum_{i=1}^L Z_{t-L+i-j}^2}{L-1}, \ \ 1\leq j\leq K.$$
For each $j\in\{1,\cdots,K\}$, by the law of large numbers (Theorem \ref{thm:LLN}), $\sigma_j^2$ can be regarded as the maximally distributed random variable on $[\ls^2,\os^2]$ when $L$ is large enough. In fact, the error term of such approximation can be estimated by (\ref{rate}).

By Theorem \ref{thmle2}, we obtain the MLE of $\os^2$ and $\ls^2$ by
$$\hat{\os}^2=\max_{1\leq j\leq K}\sigma^2_j; \ \ \hat{\ls}^2=\min_{1\leq j\leq K}\sigma^2_j.$$
Based on such estimations of upper and lower variance, extensive experiments on both NASDAQ Composite Index and S\&P 500 Index demonstrate the excellent performances of $G$-VaR model, a new benchmark predictor for value-at-risk based on the so-called $G$-normal distribution, which is superior to most existing benchmark VaR predictors (see \cite{pyy} and \cite{py}).

\begin{remark}
Theorem \ref{thmle} or Jin and Peng's optimal unbiased estimator can not be applied to this situation since the sequence $\{\sigma^2_j\}_{j=1}^K$  is not independent. In fact, it is $L$-dependent in sublinear case. Our results provide the theoretical foundation  for the widely used estimator of upper and lower variances in finance.
\end{remark}



\end{document}